\documentclass[reqno, 12pt]{amsart}

\usepackage{amsmath,amssymb,amsthm,amsfonts,mathrsfs}
\usepackage{fullpage}
\usepackage{xcolor}
\usepackage[colorlinks=true,
linkcolor=blue,
anchorcolor=blue,
citecolor=red
]{hyperref}
\allowdisplaybreaks 
\newtheorem{theorem}{Theorem}[section]
\newtheorem{lemma}[theorem]{Lemma}

\newtheorem*{conjecture*}{Conjecture}
\theoremstyle{definition}
\newtheorem{definition}[theorem]{Definition}

\theoremstyle{remark}

\newtheorem*{remark*}{remark}

\usepackage{colonequals}

\author{Runbo Li}
\address{International Curriculum Center, The High School Affiliated to Renmin University of China, Beijing, China}
\email{runbo.li.carey@gmail.com}

\makeatletter
\@namedef{subjclassname@2020}{\textup{}2020 Mathematics Subject Classification}
\makeatother

\title[]{A remark on the distribution of $\sqrt{p}$ modulo one involving primes of special type II}

\begin{document}
	
\begin{abstract}
Let $P_{r}$ denote an integer with at most $r$ prime factors counted with multiplicity. In this paper we prove that for some $\lambda < \frac{1}{12}$, the inequality $\{\sqrt{p}\}<p^{-\lambda}$ has infinitely many solutions in primes $p$ such that $p+2=P_r$, where $r= 4, 5, 6, 7$. Specially, when $r = 4$ we obtain $\lambda = \frac{1}{15.1}$, which improves Cai's $\frac{1}{15.5}$.
\end{abstract}

\maketitle

\tableofcontents

\section{Introduction}
Let $[x]$ denote the largest integer not greater than $x$ and write $\{x\} = x - [x]$. Beginning with Vinogradov \cite{vinogradov1976}, many mathematicians have studied the inequality $\{\sqrt{p}\}<p^{-\lambda}$ with prime solutions. Now the best result is due to Harman and Lewis \cite{Harman2001GaussianPI}. In \cite{Harman2001GaussianPI} they proved that there are infinitely many solutions
in primes $p$ to the inequality $\{\sqrt{p}\}<p^{-\lambda}$ with $\lambda=0.262$, which improved the previous results of Vinogradov \cite{vinogradov1976}, Kaufman \cite{kaufman1979}, Harman \cite{harman1983} and Balog \cite{balog1983}.

On the other hand, one of the famous problems in prime number theory is the twin primes problem, which states that there are infinitely many primes $p$ such that $p+2$ is also a prime. Let $P_{r}$ denote an integer with at most $r$ prime factors counted with multiplicity. Now the best result in this aspect is due to Chen \cite{Chen1973}, who showed that there are infinitely many primes $p$ such that $p+2=P_2$.

In 2013, Cai \cite{Cai1+4} combined those two problems and considered a mixed version.
\begin{definition}
Let $M(\lambda, r)$ denotes the following statement: The inequality
\begin{equation}
\{\sqrt{p}\}<p^{-\lambda}
\end{equation}
holds for infinitely many primes $p$ such that $p+2=P_r$.
\end{definition}
In his paper \cite{Cai1+4}, he also showed that
\begin{theorem}\label{Cai}
$M(\frac{1}{15.5}, 4)$ holds true.
\end{theorem}

In 2017, Dunn \cite{Dunn} considered a similar problem and improved Cai's result concerning the number of prime divisors of $p + 2$. Let $\alpha, \beta \in \mathbb{R}$ with $\alpha \neq 0$, and let $\|x\|$ denote the distance from $x$ to the nearest integer. He obtained that if $0 < \gamma < 1$ and $\theta < \frac{\gamma}{10}$, then there are infinitely many primes $p$ such that
$$
\|\alpha p^{\gamma} + \beta \| < p^{- \theta} \quad \text{and} \quad p + 2 = P_3.
$$

In 2024, Li \cite{LRBsqrtp1} generalized Cai's result to a wider range of $\lambda$. He got
\begin{theorem}\label{Li}
$M(\lambda, \lfloor \frac{8}{1-4\lambda} \rfloor)$ holds true for all $0<\lambda<\frac{1}{4}$.
\end{theorem}
In \cite{LRBsqrtp1}, Li mentioned that Cai \cite{Cai1+4} actually prove a new mean value theorem (see [\cite{Cai1+4}, Lemma 5]) for this problem and it may be useful on improving the results. In the present paper, we shall make use of this mean value theorem and improve previous results.
\begin{theorem}\label{t1}
$M(\frac{1}{15.1}, 4)$, $M(\frac{1}{12.4}, 5)$, $M(\frac{1}{12.03}, 6)$ and $M(\frac{1}{12.01}, 7)$ hold true.
\end{theorem}
We mention that $\lambda = \frac{1}{12}$ is near the limit of our method that we will explain later.

\section{Preliminary lemmas}
Let $\mathcal{A}$ denote a finite set of positive integers and $z \geqslant 2$. For square-free $d$, put
$$
\mathcal{P}=\{p : (p, 2)=1\},\quad
\mathcal{P}(r)=\{p : p \in \mathcal{P},\ (p, r)=1\},
$$
$$
P(z)=\prod_{\substack{p\in \mathcal{P}\\p<z}} p,\quad
\mathcal{A}_{d}=\{a : a \in \mathcal{A},\ d \mid a\},\quad
S(\mathcal{A}; \mathcal{P},z)=\sum_{\substack{a \in \mathcal{A} \\ (a, P(z))=1}} 1.
$$

\begin{lemma}\label{sieve} ([\cite{Rosser}, Pages 205--209]). Suppose that every $\left|\mathcal{A}_{d}\right|$ can be written as
$$
\left|\mathcal{A}_{d}\right| = \frac{\omega(d)}{d} X_{\mathcal{A}} + \eta(X_{\mathcal{A}}, d),
$$
where $\omega(d)$ is a multiplicative function, $0 \leqslant \omega(p)<p, X_{\mathcal{A}}>1$ is independent of $d$. Assume further that
$$
\sum_{z_{1} \leqslant p<z_{2}} \frac{\omega(p)}{p}=\log \frac{\log z_{2}}{\log z_{1}}+O\left(\frac{1}{\log z_{1}}\right), \quad z_{2}>z_{1} \geqslant 2.
$$
Then
$$
S(\mathcal{A}; \mathcal{P}, z) \geqslant X_{\mathcal{A}} W(z)\left\{f\left(\frac{\log D}{\log z}\right)+O\left(\frac{1}{\log ^{\frac{1}{3}} D}\right)\right\}-\sum_{\substack{d \leqslant D \\ d \mid P(z)}}\left|\eta(X_{\mathcal{A}}, d)\right|,
$$
$$
S(\mathcal{A}; \mathcal{P}, z) \leqslant X_{\mathcal{A}} W(z)\left\{F\left(\frac{\log D}{\log z}\right)+O\left(\frac{1}{\log ^{\frac{1}{3}} D}\right)\right\}+\sum_{\substack{d \leqslant D \\ d \mid P(z)}}\left|\eta(X_{\mathcal{A}}, d)\right|,
$$
where $D$ is a power of $z$,
$$
W(z) = \prod_{p \mid P(z)}\left(1-\frac{\omega(p)}{p}\right),
$$
and $f(s)$ and $F(s)$ are determined by the following differential-difference equation
\begin{align*}
\begin{cases}
F(s)=\frac{2 e^{\gamma}}{s}, \quad f(s)=0, \quad &0<s \leqslant 2,\\
(s F(s))^{\prime}=f(s-1), \quad(s f(s))^{\prime}=F(s-1), \quad &s \geqslant 2 .
\end{cases}
\end{align*}
\end{lemma}

\begin{lemma}\label{mean} ([\cite{Cai1+4}, Lemma 4]). For any given constant $A>0$ and $0<\lambda<\frac{1}{4}, 0<$ $\theta<\frac{1}{4}-\lambda$ we have
$$
\sum_{d \leqslant x^{\theta}} \max _{(l, d)=1}\left|\sum_{\substack{x < p \leqslant 2x \\ \{\sqrt{p}\}<p^{-\lambda} \\ p \equiv l(\bmod d) }} 1-\frac{(2x)^{1-\lambda}-x^{1-\lambda}}{\varphi(d)(1-\lambda)\log x} \right| \ll \frac{x^{1-\lambda}}{\log ^{A} x}.
$$
\end{lemma}

\begin{lemma}\label{newmean} ([\cite{Cai1+4}, Lemma 5]). 
Let
$$
\mathcal{N} = \left\{ p_1 p_2 p_3 p_4 m : x^{\frac{1}{14}} \leqslant p_1 < p_2 < p_3 < p_4,\ x < p_1 p_2 p_3 p_4 m \leqslant 2 x,\ \left(m, P(p_4)\right) = 1 \right\}.
$$
Then for any given constant $A>0$ and $0<\lambda<\frac{1}{8}, 0<$ $\theta<\frac{1}{4}-\lambda$ we have
$$
\sum_{d \leqslant x^{\theta}} \max _{(l, d)=1}\left| \sum_{\substack{n \in \mathcal{N} \\ n \equiv l (\bmod d) \\ \{\sqrt{n - 2}\}<(n - 2)^{-\lambda} }} 1-\frac{1}{\varphi(d)} \sum_{\substack{n \in \mathcal{N} \\ (n, d) = 1 }} n^{-\lambda} \right| \ll \frac{x^{1-\lambda}}{\log ^{A} x}.
$$
Moreover, the lower bound $x^{\frac{1}{14}}$ for prime variables can be replaced by $x^{\frac{1}{12}}$, and the proof is similar to that in \cite{Cai1+4}.
\end{lemma}

\begin{lemma}\label{Buchstab}
Let
$$
z=x^{\frac{1}{u}}, \quad 0 \leqslant y \leqslant x, \quad Q(z)=\prod_{p<z} p.
$$
Then for $u > 1$, we have
$$
\sum_{\substack{x < n \leqslant x + y \\(n, Q(z))=1}} 1 = (1+o(1)) \omega(u) \frac{y}{\log z},
$$
where $\omega(u)$ is the Buchstab function determined by the following differential-difference equation
\begin{align*}
\begin{cases}
\omega(u)=\frac{1}{u}, & \quad 1 \leqslant u \leqslant 2, \\
(u \omega(u))^{\prime}= \omega(u-1), & \quad u \geqslant 2 .
\end{cases}
\end{align*}
\end{lemma}
\begin{proof}
Lemma~\ref{Buchstab} can be proved by Prime Number Theorem with Vinogradov's error term and the inductive arguments in [\cite{HarmanBOOK}, Chapter A.2].
\end{proof}

\section{Proof of Theorem 1.4}
In this section, we define the function $\omega$ as $\omega(p)=0$ for $p = 2$ and $\omega(p)=\frac{p}{p-1}$ for other primes. Note that every odd, square-free $d$ can be written as $d = q_1 q_2 \cdots q_n$ with prime factors $q_i > 2$, we have
\begin{equation}
\frac{\omega(d)}{d} = \frac{\frac{q_1 q_2 \cdots q_n}{(q_1 - 1)(q_2 - 1) \cdots (q_n-1)}}{q_1 q_2 \cdots q_n} = \frac{1}{(q_1 - 1)(q_2 - 1) \cdots (q_n-1)} = \frac{1}{\varphi(d)}.
\end{equation}
Put
$$
D=x^{\frac{1}{4}-\lambda-\varepsilon}, \quad \mathcal{A}=\left\{p+2 : x<p \leqslant 2 x,\ \{\sqrt{p}\}<p^{-\lambda}\right\},
$$
$$
\mathcal{M} = \left\{ p_1 p_2 \cdots p_r m_1 : x^{\frac{1}{12}} \leqslant p_1 < p_2 < \cdots < p_r,\ x < p_1 p_2 \cdots p_r m_1 \leqslant 2 x,\ \left(m_1, P(p_r)\right) = 1 \right\},
$$
$$
\mathcal{B}^{1}=\left\{n-2 : n \in \mathcal{N},\ \{\sqrt{n - 2}\}<(n - 2)^{-\lambda} \right\},
$$
$$
\mathcal{B}^{2}=\left\{n-2 : n \in \mathcal{M},\ \{\sqrt{n - 2}\}<(n - 2)^{-\lambda} \right\}.
$$
Let $\gamma$ denote Euler's constant, $4 \leqslant r \leqslant 7$ and $S_r$ denote the number of prime solutions to the inequality (1) such that $p+2=P_r$, then we have
\begin{align}
\nonumber S_4 \geqslant&\ S\left(\mathcal{A}; \mathcal{P}, x^{\frac{1}{14}}\right) - \sum_{x^{\frac{1}{14}} \leqslant p_1 < p_2 < p_3 < p_4 < \left(\frac{2x}{p_1 p_2 p_3}\right)^{\frac{1}{2}}} S\left(\mathcal{A}_{p_1 p_2 p_3 p_4}; \mathcal{P}(p_1 p_2 p_3), p_4\right) + O\left(x^{\frac{13}{14}}\right) \\
=&\ S_{4, 1} - S_{4, 2} + O\left(x^{\frac{13}{14}}\right),
\end{align}
and
\begin{align}
\nonumber S_r \geqslant&\ S\left(\mathcal{A}; \mathcal{P}, x^{\frac{1}{12}}\right) - \sum_{x^{\frac{1}{12}} \leqslant p_1 < \cdots < p_r < \left(\frac{2x}{p_1 \cdots p_{r-1}}\right)^{\frac{1}{2}}} S\left(\mathcal{A}_{p_1 \cdots p_r}; \mathcal{P}(p_1 \cdots p_{r-1}), p_r\right) + O\left(x^{\frac{11}{12}}\right) \\
=&\ S_{r, 1} - S_{r, 2} + O\left(x^{\frac{11}{12}}\right)
\end{align}
for $5 \leqslant r \leqslant 7$.

In order to get a lower bound for $S_{r}$, we need to get a lower bound for $S_{r, 1}$ and an upper bound for $S_{r, 2}$. Now we ignore the presence of $\varepsilon$ for clarity.

\subsection{The evaluation of $S_{r, 1}$}
We take
\begin{equation}
X_{\mathcal{A}}=\frac{(2x)^{1-\lambda}-x^{1-\lambda}}{(1-\lambda)\log x}.
\end{equation}
Now, by (2) and the definition of $\eta(X_{\mathcal{A}}, d)$ in Lemma~\ref{sieve}, we have
\begin{align}
\nonumber \eta(X_{\mathcal{A}}, d) =&\ \left|\mathcal{A}_d \right|-\frac{\omega(d)}{d} X_\mathcal{A} \\
\nonumber =&\ \sum_{\substack{a \in \mathcal{A} \\ d \mid a}} 1-\frac{1}{\varphi(d)} X_\mathcal{A} \\
=&\ \sum_{\substack{x < p \leqslant 2x \\ \{\sqrt{p}\}<p^{-\lambda} \\ p \equiv -2 (\bmod d) }} 1-\frac{(2x)^{1-\lambda}-x^{1-\lambda}}{\varphi(d)(1-\lambda)\log x}.
\end{align}
By Lemma~\ref{mean} and (6), we can easily show that
\begin{equation}
\sum_{\substack{d \leqslant D \\ d \mid P(x^{\frac{1}{14}})}}\left|\eta(X_{\mathcal{A}}, d)\right| \ll \sum_{d \leqslant D}\mu^{2}(d) \left|\eta(X_{\mathcal{A}}, d)\right| \ll x^{1-\lambda}(\log x)^{-5}.
\end{equation}

We know that
\begin{align}
\nonumber W(z) =&\ \prod_{p \mid P(z)}\left(1-\frac{\omega(p)}{p}\right) \\
\nonumber =&\ \left(1 + O\left(\frac{1}{\log z} \right) \right) \frac{e^{-\gamma}}{\log z} \cdot \prod_{p}\left(1-\frac{\omega(p)}{p}\right) \left(1 - \frac{1}{p} \right)^{-1} \\
\nonumber =&\ \left(1 + O\left(\frac{1}{\log z} \right) \right) \frac{e^{-\gamma}}{\log z} \cdot 2 \prod_{p > 2}\left(\frac{p-2}{p-1}\right) \left(\frac{p}{p-1}\right) \\
=&\ (1+o(1)) 2 C_2 \frac{e^{-\gamma}}{\log z},
\end{align}
where
\begin{equation}
C_2:=\prod_{p>2}\left(1-\frac{1}{(p-1)^{2}}\right).
\end{equation}
Hence
\begin{equation}
W\left(x^{\frac{1}{14}}\right) = (1+o(1)) 2 C_2 \frac{e^{-\gamma}}{\frac{1}{14} \log x}.
\end{equation}

Then by Lemma~\ref{sieve} and (7)--(10), we have
\begin{align}
\nonumber S_{4, 1} \geqslant&\ X_{\mathcal{A}} W\left(x^{\frac{1}{14}}\right)\left\{f\left(\frac{\log D}{\log x^{\frac{1}{14}}}\right)+O\left(\frac{1}{\log ^{\frac{1}{3}} D}\right)\right\}-\sum_{\substack{d \leqslant D \\ d \mid P\left(x^{\frac{1}{14}}\right)}}\left|\eta(X_{\mathcal{A}}, d)\right| \\
\nonumber \geqslant&\ (1+o(1)) X_{\mathcal{A}} \cdot \frac{2 C_2 e^{-\gamma}}{\frac{1}{14} \log x} \cdot f\left(\frac{\frac{1}{4}-\lambda}{\frac{1}{14}}\right) \\
=&\ (1+o(1))\frac{2 C_2 X_{\mathcal{A}}}{\log D} \cdot \frac{e^{-\gamma}}{\left(\frac{1}{14} / \left( \frac{1}{4}-\lambda \right)\right)} f\left(\frac{\frac{1}{4}-\lambda}{\frac{1}{14}}\right).
\end{align}
Similarly, for $5 \leqslant r \leqslant 7$ we have
\begin{align}
\nonumber S_{r, 1} \geqslant&\ X_{\mathcal{A}} W\left(x^{\frac{1}{12}}\right)\left\{f\left(\frac{\log D}{\log x^{\frac{1}{12}}}\right)+O\left(\frac{1}{\log ^{\frac{1}{3}} D}\right)\right\}-\sum_{\substack{d \leqslant D \\ d \mid P\left(x^{\frac{1}{12}}\right)}}\left|\eta(X_{\mathcal{A}}, d)\right| \\
\nonumber \geqslant&\ (1+o(1)) X_{\mathcal{A}} \cdot \frac{2 C_2 e^{-\gamma}}{\frac{1}{12} \log x} \cdot f\left(\frac{\frac{1}{4}-\lambda}{\frac{1}{12}}\right) \\
=&\ (1+o(1))\frac{2 C_2 X_{\mathcal{A}}}{\log D} \cdot \frac{e^{-\gamma}}{\left(\frac{1}{12} / \left( \frac{1}{4}-\lambda \right)\right)} f\left(\frac{\frac{1}{4}-\lambda}{\frac{1}{12}}\right).
\end{align}

\subsection{The evaluation of $S_{r, 2}$} We first consider the case $r = 4$. By Chen's switching principle \cite{Chen1973}, we have
\begin{align}
\nonumber S_{4, 2} =&\ \sum_{x^{\frac{1}{14}} \leqslant p_1 < p_2 < p_3 < p_4 < \left(\frac{2x}{p_1 p_2 p_3}\right)^{\frac{1}{2}}} S\left(\mathcal{A}_{p_1 p_2 p_3 p_4}; \mathcal{P}(p_1 p_2 p_3), p_4\right) \\
=&\ S\left(\mathcal{B}^{1}; \mathcal{P}, (2x)^{\frac{1}{2}}\right).
\end{align}
The equation (13) comes from a simple observation: $S_{r, 2}$ counts the number of primes $p$ such that $p + 2 = n$ with $n \in \mathcal{N}$. Hence we have $p = n - 2$, and we can count "$n - 2$ that is prime" instead of "primes of the form $n - 2$". Now $S\left(\mathcal{B}^{1}; \mathcal{P}, (2x)^{\frac{1}{2}}\right)$ counts $n - 2$ with all prime factors larger than $(2x)^{\frac{1}{2}}$. If $n - 2$ has two or more prime factors, then their product will larger than $2x$, leading to a contradiction. Thus, the counted $n - 2$ must be prime, and the two sums are equal.

Since we have
$$
S\left(\mathcal{B}^{1}; \mathcal{P}, z\right) \leqslant S\left(\mathcal{B}^{1}; \mathcal{P}, w\right)
$$
for $w \leqslant z$, we have
\begin{equation}
S_{4, 2} = S\left(\mathcal{B}^{1}; \mathcal{P}, (2x)^{\frac{1}{2}}\right) \leqslant S\left(\mathcal{B}^{1}; \mathcal{P}, D^{\frac{1}{2}}\right).
\end{equation}

We take
\begin{equation}
X_{\mathcal{B}^{1}} = \sum_{n \in \mathcal{N}} n^{-\lambda}.
\end{equation}
Now, by (2) and the definition of $\eta(X_{\mathcal{A}}, d)$ in Lemma~\ref{sieve}, we have
\begin{align}
\nonumber \eta(X_{\mathcal{B}^{1}}, d) =&\ \left|\mathcal{B}^{1}_d \right|-\frac{\omega(d)}{d} X_{\mathcal{B}^{1}} \\
\nonumber =&\ \sum_{\substack{b \in \mathcal{B}^{1} \\ d \mid b}} 1-\frac{1}{\varphi(d)} X_{\mathcal{B}^{1}} \\
\nonumber =&\ \sum_{\substack{n \in \mathcal{N} \\ n \equiv 2 (\bmod d) \\ \{\sqrt{n - 2}\}<(n - 2)^{-\lambda} }} 1-\frac{1}{\varphi(d)} \sum_{n \in \mathcal{N}} n^{-\lambda} \\
\nonumber =&\ \sum_{\substack{n \in \mathcal{N} \\ n \equiv 2 (\bmod d) \\ \{\sqrt{n - 2}\}<(n - 2)^{-\lambda} }} 1-\frac{1}{\varphi(d)} \sum_{\substack{n \in \mathcal{N} \\ (n, d) = 1 }} n^{-\lambda} \\
\nonumber & + \sum_{\substack{n \in \mathcal{N} \\ n \equiv 2 (\bmod d) \\ \{\sqrt{n - 2}\}<(n - 2)^{-\lambda} }} 1-\frac{1}{\varphi(d)} \sum_{\substack{n \in \mathcal{N} \\ (n, d) > 1 }} n^{-\lambda} \\
=&\ \eta_{1}(X_{\mathcal{B}^{1}}, d) + \eta_{2}(X_{\mathcal{B}^{1}}, d).
\end{align}
Applying Lemma~\ref{newmean} directly, we can show that
\begin{equation}
\sum_{\substack{d \leqslant D \\ d \mid P\left(D^{\frac{1}{2}}\right)}}\left|\eta_{1}(X_{\mathcal{B}^{1}}, d)\right| \ll \sum_{d \leqslant D}\mu^{2}(d) \left|\eta_{1}(X_{\mathcal{B}^{1}}, d)\right| \ll x^{1-\lambda}(\log x)^{-5}.
\end{equation}
The sum of $\eta_{2}(X_{\mathcal{B}^{1}}, d)$ can be bounded trivially:
\begin{equation}
\sum_{\substack{d \leqslant D \\ d \mid P\left(D^{\frac{1}{2}}\right)}}\left|\eta_{2}(X_{\mathcal{B}^{1}}, d)\right| \ll x^{1-\frac{1}{14}}\log x.
\end{equation}
When $\lambda = \frac{1}{15.1}$, we have
\begin{align}
\nonumber \sum_{\substack{d \leqslant D \\ d \mid P\left(D^{\frac{1}{2}}\right)}}\left|\eta(X_{\mathcal{B}^{1}}, d)\right| \ll&\ x^{1-\lambda}(\log x)^{-5} + x^{1-\frac{1}{14}}\log x \\
\ll&\ x^{1-\lambda}(\log x)^{-5}.
\end{align}

Then by Lemma~\ref{sieve}, (8) and (19), we have
\begin{align}
\nonumber S_{4, 2} \leqslant&\ X_{\mathcal{B}^{1}} W\left(D^{\frac{1}{2}}\right)\left\{F\left(\frac{\log D}{\log D^{\frac{1}{2}}}\right)+O\left(\frac{1}{\log ^{\frac{1}{3}} D}\right)\right\}+\sum_{\substack{d \leqslant D \\ d \mid P\left(D^{\frac{1}{2}}\right)}}\left|\eta(X_{\mathcal{B}^{1}}, d)\right| \\
\nonumber \leqslant&\ (1+o(1)) X_{\mathcal{B}^{1}} \cdot \frac{2 C_2 e^{-\gamma}}{\frac{1}{2} \log D} \cdot F(2) \\
=&\ (1+o(1))\frac{4 C_2 X_{\mathcal{B}^{1}}}{\log D}.
\end{align}

By Lemma~\ref{Buchstab}, Prime Number Theorem and integration by parts we have
\begin{equation}
X_{\mathcal{B}^{1}} = (1+o(1)) X_{\mathcal{A}} T_4,
\end{equation}
where
\begin{equation}
T_4 = \int_{\frac{1}{14}}^{\frac{1}{5}} \int_{t_1}^{\frac{1-t_1}{4}} \int_{t_2}^{\frac{1-t_1-t_2}{3}} \int_{t_3}^{\frac{1-t_1-t_2-t_3}{2}} \frac{\omega \left( \frac{1-t_1-t_2-t_3-t_4}{t_4} \right)}{t_1 t_2 t_3 t_4^2} d t_4 d t_3 d t_2 d t_1,
\end{equation}
where $\omega(u)$ is defined in Lemma~\ref{Buchstab}.

Combining (3), (11), (20) and (21), we have
\begin{equation}
S_4 \geqslant (1+o(1))\frac{2 C_2 X_{\mathcal{A}}}{\log D} \left( \frac{e^{-\gamma}}{\left(\frac{1}{14} / \left( \frac{1}{4}-\lambda \right)\right)} f\left(\frac{\frac{1}{4}-\lambda}{\frac{1}{14}}\right) - 2 T_4 \right).
\end{equation}
Hence we only need
\begin{equation}
\frac{e^{-\gamma}}{\left(\frac{1}{14} / \left( \frac{1}{4}-\lambda \right)\right)} f\left(\frac{\frac{1}{4}-\lambda}{\frac{1}{14}}\right) - 2 T_4 > 0.
\end{equation}
Numerical calculation shows that (24) holds for $\lambda = \frac{1}{15.1}$, hence $M(\frac{1}{15.1}, 4)$ holds true.

Similarly, for the case $5 \leqslant r \leqslant 7$ we have
\begin{align}
\nonumber S_{r, 2} =&\ \sum_{x^{\frac{1}{12}} \leqslant p_1 < \cdots < p_r < \left(\frac{2x}{p_1 \cdots p_{r-1}}\right)^{\frac{1}{2}}} S\left(\mathcal{A}_{p_1 \cdots p_r}; \mathcal{P}(p_1 \cdots p_{r-1}), p_r\right) \\
=&\ S\left(\mathcal{B}^{2}; \mathcal{P}, (2x)^{\frac{1}{2}}\right) \leqslant S\left(\mathcal{B}^{2}; \mathcal{P}, D^{\frac{1}{2}}\right).
\end{align}

We take
\begin{equation}
X_{\mathcal{B}^{2}} = \sum_{n \in \mathcal{M}} n^{-\lambda}.
\end{equation}
Now, by (2) and the definition of $\eta(X_{\mathcal{A}}, d)$ in Lemma~\ref{sieve}, we have
\begin{align}
\nonumber \eta(X_{\mathcal{B}^{2}}, d) =&\ \left|\mathcal{B}^{2}_d \right|-\frac{\omega(d)}{d} X_{\mathcal{B}^{2}} \\
\nonumber =&\ \sum_{\substack{b \in \mathcal{B}^{2} \\ d \mid b}} 1-\frac{1}{\varphi(d)} X_{\mathcal{B}^{2}} \\
\nonumber =&\ \sum_{\substack{n \in \mathcal{M} \\ n \equiv 2 (\bmod d) \\ \{\sqrt{n - 2}\}<(n - 2)^{-\lambda} }} 1-\frac{1}{\varphi(d)} \sum_{n \in \mathcal{M}} n^{-\lambda} \\
\nonumber =&\ \sum_{\substack{n \in \mathcal{M} \\ n \equiv 2 (\bmod d) \\ \{\sqrt{n - 2}\}<(n - 2)^{-\lambda} }} 1-\frac{1}{\varphi(d)} \sum_{\substack{n \in \mathcal{M} \\ (n, d) = 1 }} n^{-\lambda} \\
\nonumber & + \sum_{\substack{n \in \mathcal{M} \\ n \equiv 2 (\bmod d) \\ \{\sqrt{n - 2}\}<(n - 2)^{-\lambda} }} 1-\frac{1}{\varphi(d)} \sum_{\substack{n \in \mathcal{M} \\ (n, d) > 1 }} n^{-\lambda} \\
=&\ \eta_{1}(X_{\mathcal{B}^{2}}, d) + \eta_{2}(X_{\mathcal{B}^{2}}, d).
\end{align}
Taking $m = p_5 \cdots p_r m_1$ in Lemma~\ref{newmean}, we know that conditions
$$
p_4 < p_5 < \cdots < p_r \quad \text{and} \quad \left(m, P(p_4)\right)
$$
are fulfilled. By Lemma~\ref{newmean} (with $x^{\frac{1}{14}}$ replaced by $x^{\frac{1}{12}}$), we can show that
\begin{equation}
\sum_{\substack{d \leqslant D \\ d \mid P\left(D^{\frac{1}{2}}\right)}}\left|\eta_{1}(X_{\mathcal{B}^{2}}, d)\right| \ll \sum_{d \leqslant D}\mu^{2}(d) \left|\eta_{1}(X_{\mathcal{B}^{2}}, d)\right| \ll x^{1-\lambda}(\log x)^{-5}.
\end{equation}
The sum of $\eta_{2}(X_{\mathcal{B}^{2}}, d)$ can be bounded trivially:
\begin{equation}
\sum_{\substack{d \leqslant D \\ d \mid P\left(D^{\frac{1}{2}}\right)}}\left|\eta_{2}(X_{\mathcal{B}^{2}}, d)\right| \ll x^{1-\frac{1}{12}}\log x.
\end{equation}
When $\lambda < \frac{1}{12}$, we have
\begin{align}
\nonumber \sum_{\substack{d \leqslant D \\ d \mid P\left(D^{\frac{1}{2}}\right)}}\left|\eta(X_{\mathcal{B}^{1}}, d)\right| \ll&\ x^{1-\lambda}(\log x)^{-5} + x^{1-\frac{1}{12}}\log x \\
\ll&\ x^{1-\lambda}(\log x)^{-5}.
\end{align}

Then by Lemma~\ref{sieve}, (8) and (30), for $5 \leqslant r \leqslant 7$ we have
\begin{align}
\nonumber S_{r, 2} \leqslant&\ X_{\mathcal{B}^{2}} W\left(D^{\frac{1}{2}}\right)\left\{F\left(\frac{\log D}{\log D^{\frac{1}{2}}}\right)+O\left(\frac{1}{\log ^{\frac{1}{3}} D}\right)\right\}+\sum_{\substack{d \leqslant D \\ d \mid P\left(D^{\frac{1}{2}}\right)}}\left|\eta(X_{\mathcal{B}^{2}}, d)\right| \\
\nonumber \leqslant&\ (1+o(1)) X_{\mathcal{B}^{2}} \cdot \frac{2 C_2 e^{-\gamma}}{\frac{1}{2} \log D} \cdot F(2) \\
=&\ (1+o(1))\frac{4 C_2 X_{\mathcal{B}^{2}}}{\log D}.
\end{align}

Similar to the case $r = 4$, by Lemma~\ref{Buchstab}, Prime Number Theorem and integration by parts we have
\begin{equation}
X_{\mathcal{B}^{2}} = (1+o(1)) X_{\mathcal{A}} T_r,
\end{equation}
where
\begin{equation}
T_r = \int_{\frac{1}{12}}^{\frac{1}{r+1}} \int_{t_1}^{\frac{1-t_1}{r}} \cdots \int_{t_{r-1}}^{\frac{1-t_1-\cdots-t_{r-1}}{2}} \frac{\omega \left( \frac{1-t_1-\cdots-t_r}{t_r} \right)}{t_1 t_2 \cdots t_{r-1} t_r^2} d t_r \cdots d t_1.
\end{equation}

Combining (4), (12), (31) and (32), for $5 \leqslant r \leqslant 7$ we have
\begin{equation}
S_r \geqslant (1+o(1))\frac{2 C_2 X_{\mathcal{A}}}{\log D} \left( \frac{e^{-\gamma}}{\left(\frac{1}{12} / \left( \frac{1}{4}-\lambda \right)\right)} f\left(\frac{\frac{1}{4}-\lambda}{\frac{1}{12}}\right) - 2 T_r \right).
\end{equation}
Hence we only need
\begin{equation}
\frac{e^{-\gamma}}{\left(\frac{1}{12} / \left( \frac{1}{4}-\lambda \right)\right)} f\left(\frac{\frac{1}{4}-\lambda}{\frac{1}{12}}\right) - 2 T_r > 0.
\end{equation}
When $r= 5, 6, 7$, numerical calculation shows that
$$
\frac{e^{-\gamma}}{\left(\frac{1}{12} / \left( \frac{1}{4}-\frac{1}{12.4} \right)\right)} f\left(\frac{\frac{1}{4}-\frac{1}{12.4}}{\frac{1}{12}}\right) - 2 T_5 > 0,
$$
$$
\frac{e^{-\gamma}}{\left(\frac{1}{12} / \left( \frac{1}{4}-\frac{1}{12.03} \right)\right)} f\left(\frac{\frac{1}{4}-\frac{1}{12.03}}{\frac{1}{12}}\right) - 2 T_6 > 0
$$
and
$$
\frac{e^{-\gamma}}{\left(\frac{1}{12} / \left( \frac{1}{4}-\frac{1}{12.01} \right)\right)} f\left(\frac{\frac{1}{4}-\frac{1}{12.01}}{\frac{1}{12}}\right) - 2 T_7 > 0.
$$
Now Theorem~\ref{t1} is proved. We remark that for positive $\lambda$, we have
\begin{equation}
f\left(\frac{\frac{1}{4}-\lambda}{\lambda}\right) > 0 \quad \text{or} \quad \frac{\frac{1}{4}-\lambda}{\lambda} > 2
\end{equation}
only when $\lambda < \frac{1}{12}$, so $\lambda = \frac{1}{12.01}$ is rather near the limit point.

\bibliographystyle{plain}
\bibliography{bib}

\begin{thebibliography}{10}

\bibitem{balog1983}
A.~Balog.
\newblock On the fractional part of $p^{\theta}$.
\newblock {\em Archiv der Mathematik}, 40:434--440, 1983.

\bibitem{Cai1+4}
Y.~Cai.
\newblock On the distribution of {$\sqrt{p}$} modulo one involving primes of special type.
\newblock {\em Studia Scientiarum Mathematicarum Hungarica}, 50(4):470--490, 2013.

\bibitem{Chen1973}
J.~R. Chen.
\newblock On the representation of a larger even integer as the sum of a prime and the product of at most two primes.
\newblock {\em Sci. Sinica}, 16:157--176, 1973.

\bibitem{Dunn}
A.~Dunn.
\newblock On the distribution of $\alpha p^{\gamma} + \beta$ modulo one.
\newblock {\em J. Number Theory}, 176:67--75, 2017.

\bibitem{harman1983}
G.~Harman.
\newblock On the distribution of $\sqrt{p}$ modulo one.
\newblock {\em Mathematika}, 30(1):104--116, 1983.

\bibitem{HarmanBOOK}
G.~Harman.
\newblock {\em Prime--detecting Sieves}, volume~33 of {\em London Mathematical Society Monographs (New Series)}.
\newblock Princeton University Press, Princeton, NJ, 2007.

\bibitem{Harman2001GaussianPI}
G.~Harman and P.~Lewis.
\newblock Gaussian primes in narrow sectors.
\newblock {\em Mathematika}, 48:119--135, 2001.

\bibitem{Rosser}
H.~Iwaniec.
\newblock Rosser's sieve.
\newblock In {\em Recent Progress in Analytic Number Theory II}, pages 203--230. Academic Press, 1981.

\bibitem{kaufman1979}
R.~M. Kaufman.
\newblock The distribution of $\sqrt{p}$.
\newblock {\em Matematicheskie Zametki}, 26(4):497--504, 1979.

\bibitem{LRBsqrtp1}
R.~Li.
\newblock {A remark on the distribution of $\sqrt{p}$ modulo one involving primes of special type}.
\newblock {\em Hiroshima Mathematical Journal, to appear. arXiv e-prints}, page arXiv:2401.01351v1, 2024.

\bibitem{vinogradov1976}
I.~M. Vinogradov.
\newblock Special variants of the method of trigonometric sums.
\newblock {\em Ivan Matveevich Vinogradov: Selected Works}, 1976.

\end{thebibliography}
\end{document}